\documentclass[10]{article}
\usepackage{amsmath,amssymb,amsthm}

\usepackage{mathtools}

\usepackage{enumitem}

\usepackage[textwidth=2cm,textsize=small,shadow]{todonotes}

\usepackage[letterpaper,right=3cm,left=3cm]{geometry}                % See geometry.pdf to learn the layout options. There are lots.
\usepackage{fancyhdr,lastpage}
\usepackage{optidef}

\usepackage{listings}

\usepackage{tikz}
\usepackage{pgfplots}
\pgfplotsset{compat=1.18}

\usepackage{hyperref}
\hypersetup{
  pdflang=en-US,
  pdftitle={On Nathanson's Triangular Number Phenomenon},
  pdfauthor={Dr Kevin O'Bryant}
  pdfsubject={Sumsets},
  pdfkeywords={sumsets,types of addition tables,triangular numbers,tetrahedral numbers}
}

\newcommand{\seqnum}[1]{\href{https://oeis.org/#1}{\rm \underline{#1}}}

\newcommand{\NN}{\ensuremath{\mathbb N}}
\newcommand{\ZZ}{\ensuremath{\mathbb Z}}

\newcommand{\RR}{\ensuremath{\mathbb R}}

\newcommand{\cR}{\ensuremath{{\cal R}}}
\newcommand{\cT}{\ensuremath{{\cal T}}}
\newcommand{\cL}{\ensuremath{{\cal L}}}
\newcommand{\cX}{\ensuremath{{\cal X}}}

\renewcommand{\vec}[1]{\overline{\mathbf #1}}

\makeatletter
\def\localbig#1#2{%
  \sbox\z@{$#1
    \sbox\tw@{$#1()$}%
    \dimen@=\ht\tw@\advance\dimen@\dp\tw@
    \nulldelimiterspace\z@\left#2\vcenter to1.2\dimen@{}\right.
    $}\box\z@}
\makeatletter

\newtheorem{theorem}{Theorem} %[section]
\newtheorem*{theorem*}{Theorem}
\newtheorem{lemma}[theorem]{Lemma}
\newtheorem{question}[theorem]{Question}

\newtheorem{corollary}[theorem]{Corollary}
\theoremstyle{definition}
\newtheorem{definition}[theorem]{Definition}
\newtheorem{conjecture}[theorem]{Conjecture}

\DeclareMathOperator{\diam}{diam}

\DeclareMathOperator{\sep}{sep}

\begin{document}
    \title{On Nathanson's Triangular Number Phenomenon}
    \author{Kevin O'Bryant\footnote{On faculty at the City University of New York, both at the College of Staten Island campus and The Graduate Center campus.}}
    \date{ \today }
    \maketitle

\begin{abstract}
For a finite set $A\subseteq \mathbb{Z}$, the $h$-fold sumset is $hA \coloneqq \{x_1+\dots+x_h:x_i\in A\}$. We interpret the beginning of the sequence of sumset sizes $(|hA|)_{h=1}^\infty$ in terms of the successive $L^1$-minima of a lattice (specifically, the points in $\ZZ^{|A|}$ whose coordinates sum to 0 and which are perpendicular to $\langle a_1,\dots,a_{|A|}\rangle$). In particular, if $h_1,h_2$ are the first and second minima, and $1\le h<h_1$, then $|hA|=\binom{h+|A|-1}{|A|-1}$, while if $h_1\le h <h_2$, then $|hA|=\binom{h+|A|-1}{|A|-1}-\binom{h-h_1+|A|-1}{|A|-1}$.

This explains the appearance of triangular numbers in the sequence of sumset sizes, an observation related to a recent experiment of Nathanson.
\end{abstract}

\section{Introduction}
The $h$-fold sumset of a set $A$ of $k$ integers is $hA \coloneqq \{x_1+\dots+x_h:x_i\in A\}$. By elementary counting, its size satisfies both
\[|hA| \le h \cdot \diam(A) + 1,\qquad \text{and} \qquad |hA| \le  \binom{h+k-1}{k-1} .\]
By Nathanson's Theorem~\cite{nathanson1972sumsoffinitesetsofintegers}, for sufficiently large $h$ we have $|hA|=\diam(A) \cdot h + C_0$ for some fixed $C_0\le 1$. For sufficiently small $h$, we have $|hA| = \binom{h+k-1}{k-1}$. In this work, we address the meaning of ``sufficiently small'', and give an explicit formula for $|hA|$ for $h$ that is slightly larger.

For example, consider $4$-element sets. As there are $\binom{h+3}{3}$ unordered $h$-tuples of elements of $A$, the quantity $\binom{h+3}{3}-|hA|$ counts the ``missing'' sums. Consider specifically $A=\{0, 2, 18, 25\}$ and the following table.
\[
\begin{array}{r|cccccccccccc}
h & 1 & 2 & 3 & 4 & 5 & 6 & 7 & 8 & 9 & 10 & 11 & 12 \\ \hline
|hA| & 4 & 10 & 20 & 34 & 52 & 74 & 100 & 130 & 162 & 193 & 222 & 249 \\
\binom{h+3}{3} - |hA| & 0 & 0 & 0 & 1 & 4 & 10 & 20 & 35 & 58 & 93 & 142 & 206 \\
\text{first differences}& & 0 & 0 & 1 & 3 & 6  & 10 & 15 & 23 & 35 & 49 & 64
\end{array}
\]
The appearance of the triangular numbers (\seqnum{A000217}) in the start of the bottom row was surprising, and the main theorem of this work explains its appearance. Namely, we prove that the second row of such a table begins with tetrahedral numbers (\seqnum{A000292}). The string of numbers in the first row, which are of the form $2(h^2+1)$, is also interesting, and is explained below in Corollary~\ref{cor:special poly}. In Section~\ref{sec:NathansonExperiment} below, we repeat a related experiment of Nathanson~\cite{nathanson2025triangulartetrahedralnumberdifferences} in which the triangular numbers appear. Indeed, this work was directly inspired by that work.

\subsection{Notation and Conventions}
We use $\NN$ for the positive integers, and $\NN_0$ for the set $\{0\}\cup \NN$. For each finite set $A\subseteq\RR$, we denote its elements as $a_1<\cdots<a_k$, so that $|A|=k$ and $\diam(A) = a_k-a_1$.
We use the convention that binomial coefficients are zero if their top argument is smaller than their bottom argument, and the convention that if $X$ is a set and $k$ a nonnegative integer, then $\binom{X}{k}$ is the set of subsets of $X$ with cardinality $k$.

The sumset $hA$ is the set $\{b_1+\cdots+b_h : b_i \in A\}$. Recursively, $1A\coloneqq A$ and $hA \coloneqq \{s+b:s \in (h-1)A,b\in A\}$. The size of $hA$ is invariant under translation and dilation of $A$, often enabling us to assume without loss of generality that $\min A=0$ and that $\gcd(A)=1$.

Let
\[\cX_{t,k} \coloneqq \left\{ \langle c_1,\dots,c_k\rangle : c_i \in \NN_0, c_1+\dots+c_k=t\right\}.\]
With $\vec a \coloneq \langle a_1,a_2,\dots,a_k\rangle$, we have
\[ hA = \left\{ \vec x \cdot \vec a : \vec x \in \cX_{h,k} \right\}.\]
By stars-and-bars,
\[ |\cX_{t,k}| = \binom{t+k-1}{k-1}.\]

For each set $A\subseteq\ZZ$, we let $\cL$ be the $(k-2)$-dimensional lattice in $\ZZ^k$ of vectors that are perpendicular to both $\vec 1$ (the vector of all $1$'s) and $\vec a$. This lattice, which we call \emph{the coefficient lattice of $A$}, has dimension $k-2$, and consequently for $k\ge 4$ this dimension is at least $2$. The entries of an element of $\cL$ sum to $0$ and are integers, so the sum of their absolute values is an even integer.

We define the \emph{$i$-th successive $L^1$-minima} $\lambda_i$ of $\cL$ to be the infimum of $\lambda$ such that $\{\vec c : \|\vec c\|_1 \le \lambda\}$ contains $i$ linearly independent vectors. We let $\vec{y_1},\vec{y_2},\dots,\vec{y_{k-2}}$ be (the nonuniquely defined) linearly independent vectors in $\cL$ with $\|\vec{y_i}\|_1=\lambda_i$, and call them \emph{minimizers}.

For the coefficient lattices we consider, $\lambda_i$ must be an even integer. We define $h_1,h_2$ so that $\lambda_1=2h_1,\lambda_2=2h_2$ are the first and second minima of the lattice with respect to the $L^1$-norm. That is, every nonzero vector in $\cL$ has $L^1$-norm at least $2h_1$, there is a first minimizer $\vec y_1 \in \cL$ with $\|y_1\|_1=2h_1$, and every vector in $\cL$ that is linearly independent of $\vec y$ has $L^1$-norm at least $2h_2$.

By Lemma~\ref{lem:all minima happen}, every pair of even numbers $2h_1,2h_2$ with $2\le h_1\le h_2$ arises as the successive minima of such a lattice, even with $k=4$. The proof of Lemma~\ref{lem:all minima happen} is delayed to \S\ref{subsec:they happen}.

\begin{lemma}\label{lem:all minima happen}
For every integer $k\ge 4$ and every pair of integers $2\le h_1\le h_2$, there is a $k$-element set $A$ whose coefficient lattice $\cL$ has first $L^1$-minimum $2h_1$ and second $L^1$-minimum $2h_2$.
\end{lemma}

We conjecture that for every $1\le h_1\le h_2\le \cdots \le h_{k-2}$, there is a set $A$ of $k$ integers whose coefficient lattice has successive $L^1$-minima $2h_1,\dots,2h_{k-2}$.

\clearpage
\subsection{Theorems and Corollaries}

Our Main Theorem connects the minima of the coefficient lattice $\cL$ to the sequence of sumset sizes.
\begin{theorem}\label{thm:main}
Let $A=\{a_1,\dots,a_k\}$ be a set of $k\ge 4$ integers. Let $\cL$ be the sublattice of $\ZZ^k$ that is perpendicular to both $\vec 1$ and $\vec a=\langle a_1,\dots,a_k\rangle$. Let $2h_1,2h_2$ be the first and second minima of $\cL$ with respect to the $L^1$-norm.
If $1\le h<h_1$, then
\begin{equation}\label{eq:B_h case} |hA| = \binom{h+k-1}{k-1}.\end{equation}
If $h_1 \le h < h_2$, then
\begin{equation} |hA| = \binom{h+k-1}{k-1} - \binom{h-h_1+k-1}{k-1}.\end{equation}
\end{theorem}

If $h<h_1$, then $|hA| = \binom{h+k-1}{k-1}$. That is, $A$ is a $B_h$-set. Further, for $h> h_1$
\begin{multline*}
\left(\binom{h+k-1}{k-1} - |hA|\right) - \left(\binom{(h-1)+k-1}{k-1} - |(h-1)A|\right)\\
 = \binom{h-h_1+k-1}{k-1} - \binom{(h-1)-h_1+k-1}{k-1} = \binom{h-h_1+k-2}{k-2}.
\end{multline*}
If $k=4$, with $t=h-h_1+2$, this is $\binom{t}{2}$, the triangular numbers seen in the bottom row of the table above. We comment that with $k=4$, both $\binom{h+k-1}{k-1}$ and $\binom{h-h_1+k-1}{k-1}$ are tetrahedral numbers~\seqnum{A000292}.

This tells us that the sequence of sumset sizes of a set of $k$ integers is given exactly by a degree $k-1$ polynomial for the first several $h$, and exactly by a degree $k-2$ polynomial for the next several $h$. The author suspects, but has not proven, that if $\lambda_{i-1}\le h < \lambda_i$, then $|hA|$ is ``essentially'' a degree $k-i$ polynomial whose coefficients are determined by the first $i-1$ minimizers. The vague ``essentially'' qualifier is necessary to handle congruence issues that can arise for larger $i$, as seen in the Frobenius Problem.

Returning to the example above, $A=\{0, 2, 18, 25\}$, examining the coefficient lattice $\cL$ we find that $h_1=4$ and $h_2=9$. By our theorem, $|hA| = \binom{h+3}{3}$ for $1\le h <4$, and $|hA| = \binom{h+3}{3}-\binom{h-1}{3}=2(h^2+1)$ for $4\le h <9$. That $|hA|=2(h^2+1)$ also for $1\le h \le 3$, is a coincidence arising from $\binom{h+3}{3}=2(h^2+1)$ for $1\le h \le 3$. The set $A=\{0,1,3b+1,3b+4\}$ gives us the pleasant Corollary~\ref{cor:special poly}.

\begin{corollary}\label{cor:special poly}
For all $b>1$, there is a set $A$ of integers with $|hA|=2(h^2+1)$ for $1\le h \le b$.
\end{corollary}

Call a polynomial $p(h)$ \emph{cute} if for all $b>1$, there is a set $A$ of integers with $|hA|=p(h)$ for $1\le h \le b$. We are unable to enumerate the cute polynomials.

In~\cite{nathanson2025triangulartetrahedralnumberdifferences}, Nathonson lists a number of problems concerning
  \[\cR_{\ZZ}(h,k) \coloneqq \{ |hA| : A\subseteq \ZZ, |A|=k\}.\]
Problem 8 of his program is solved by Corollary~\ref{cor:they all happen} to Theorem~\ref{thm:main}, whose proof gives an explicit set $A$ for each of the claimed sizes of $|hA|$.
\begin{corollary}\label{cor:they all happen}
Let $h$ be a positive integer. Then
\[\left\{ \binom{h+k-1}{k-1}-\binom{h+k-j}{k-1} : 1< j \le h\right\} \subseteq \cR_{\ZZ}(h,k).\]
\end{corollary}

While of a different flavor, the proof of the following result has some conceptual and stylistic overlap with our proof of Theorem~\ref{thm:main}. In~\cite{nathanson2025compressioncomplexitysumsetsizes}, Nathanson writes ``it is not known if $\cR_{\NN_0}(h,k)= \cR_{\RR}(h,k)$.'' As $\cR_{\NN_0}(h,k)=\cR_{\NN}(h,k)$ trivially, the first equality in our Theorem~\ref{thm:NvsR} fills this gap in the literature. Let $A^h \coloneqq \{a_1\cdots a_h : a_i \in A\}$ be the $h$-fold product set.

\begin{theorem}\label{thm:NvsR}
We have
\begin{multline*}
\left\{ |hA| : |A|=k,A\subseteq \NN \right\}
=\left\{ |hA| : |A|=k,A\subseteq \RR \right\}\\
=\left\{ |A^h| : |A|=k,A\subseteq \NN \right\}
=\left\{ |A^h| : |A|=k,A\subseteq \RR \right\}.
\end{multline*}
\end{theorem}

\section{Proof of Theorem~\ref{thm:main} and Corollaries}

\begin{proof}[Proof of Theorem~\ref{thm:main}]
Suppose that we have distinct $\vec {c_1},\vec{c_2}\in \cX_{h,k}$ with $\vec {c_1}\cdot \vec a = \vec {c_2}\cdot \vec a$. Then $\vec{c_1}-\vec{c_2} \in \cL$, and
  \begin{equation}\label{eq:triangle}
  2h_1 \le \| \vec{c_1}-\vec{c_2} \|_1 \le \|\vec {c_1} \|_1 + \|\vec {c_2}\|_1=2h.
  \end{equation}
If $h<h_1$, this is impossible, and so $|hA|=|\cX_{h,k}|=\binom{h+k-1}{k-1}$, as claimed.

We now assume that $h_1 \le h < h_2$, and take $y\in\cL$ with $\| y \|_1=2h_1$.

The $h$-type of $A$ is the equivalence relation on $\cX_{h,k}$ defined by $\vec {c_1} \equiv \vec{c_2}$ if $\vec{c_1} \cdot \vec a = \vec{c_2}\cdot \vec a$. The number of classes in the $h$-type of $A$ is $|hA|$.

We define a second equivalence relation on $\cX_{h,k}$ by taking $\vec{c_1} \cong \vec{c_2}$ if $\vec{c_1}-\vec{c_2}$ is a multiple of $\vec y$. We will first prove that the two equivalence relations $\cong,\equiv$ are identical, and then that the $\cong$ equivalence relation has $\binom{h+k-1}{k-1} - \binom{h-h_1+k-1}{k-1}$ classes.

First, suppose that $\vec{c_1} \equiv \vec{c_2}$ and $\vec{c_1}\neq \vec{c_2}$. Then $(\vec{c_1}-\vec{c_2})\cdot \vec a = 0$, and $(\vec{c_1}-\vec{c_2})\cdot \vec 1=\vec{c_1}\cdot\vec 1 - \vec{c_2}\cdot \vec 1= h-h =0$, so that $\vec{c_1}-\vec{c_2}\in\cL$. But since $\| \vec{c_1}-\vec{c_2}\|=2h$, and $2h<2h_2$, it must be that $\vec{c_1}-\vec{c_2}=\alpha\vec y$ for some $\alpha\in\ZZ$. Ergo, $\vec{c_1} \cong \vec{c_2}$.

Now, suppose that $\vec{c_1} \cong \vec{c_2}$ and $\vec{c_1}\neq \vec{c_2}$. Then $\vec{c_1}-\vec{c_2} = \alpha \vec y $.  Consequently, $(\vec{c_1}-\vec{c_2})\cdot \vec a = \alpha \vec y \cdot \vec a = 0$, since $\vec y \in \cL$. Thus, $\vec{c_1} \equiv \vec{c_2}$.

It remains only to count the $\cong$-classes. Each class is finite, so has a representative $\vec{c}$ with $\vec{c}-\vec{y}$ \emph{not} in the class. Every vector in the class with $\vec{c}$ has the form $\vec{c}+i\vec{y}$, and we let $\alpha$ be the maximal such $i$. The classes are convex, being the intersection of a space with the convex $\cX_{h,k}$, so the entire class is $\vec{c},\vec{c}+\vec y, \vec{c}+2\vec{y}, \dots,\vec{c}+\alpha \vec y$ for some nonnegative integer $\alpha$. Each class with $\alpha+1$ elements has exactly $\alpha$ solutions to $\vec{c_2}-\vec{c_1}=\vec{y}$. It follows\footnote{This is identical to the theorem that for a forest $G=(V,E)$, the number of connected components is $|V|-|E|$.} that the number of classes is
  \[|\cX_{h,k}|-\#\left\{(\vec{c_1},\vec{c_2}) \in \cX_{h,k}^2 : \vec{c_2}-\vec{c_1}=\vec y\right\}.\]

We now remains to count the number of $\vec{c_1},\vec{c_2} \in \cX_{h,k}$ with $\vec{c_2}-\vec{c_1}=\vec{y}$.
Let $\vec y = \langle y_1,\dots,y_k\rangle$, and define
\[\vec u = \langle u_1,\dots,u_k\rangle, u_i \coloneqq \begin{cases} 0, & y_i \le 0; \\ y_i, & y_i >0;\end{cases}
\qquad \text{ and } \qquad
\vec v = \langle v_1,\dots,v_k\rangle, v_i \coloneqq \begin{cases} -y_i, & y_i \le 0; \\ 0, & y_i >0.\end{cases}\]
We have $\|\vec u \|_1 = \|\vec v \|_1 = h_1$, since $\vec y \cdot \vec 1 =0$, and $\vec y = \vec u - \vec v$.
Every solution to $\vec{c_2}-\vec{c_1}=\vec y$ corresponds to a solution to $\vec{c_2}-\vec{u}=\vec{c_1}-\vec{v}$, and vice versa. Now consider the $i$-th coordinate of $\vec r \coloneqq \vec{c_2}-\vec{u}$. If $y_i\le 0$, then $u_i =0$, and so
\[r_i=(\vec{c_2}-\vec{u})_i=(\vec{c_2})_i-y_i=(\vec{c_2})_i,\]
which is nonnegative since $\vec{c_2}\in \cX_{h,k}$. On the other hand, if $y_i>0$, then $v_i=0$ and \[r_i=(\vec{c_2}-\vec{u})_i=(\vec{c_1}-\vec{v})_i=(\vec{c_1})_i-v_i=(\vec{c_1})_i,\] which is nonnegative since $\vec{c_1}\in \cX_{h,k}$. Thus, $\vec{c_2}-\vec{u}$
has nonnegative integer coordinates summing $h-h_1$, and so $\vec r \in \cX_{h-h_1,k}$.

That is, the solutions to $\vec{c_2}-\vec{c_1}=\vec y$ are in 1-1 correspondence with the elements of $\cX_{h-h_1,k}$. We now have
\[ \#\left\{(\vec{c_1},\vec{c_2}) \in \cX_{h,k}^2 : \vec{c_2}-\vec{c_1}=\vec y\right\}
 = |\cX_{h-h_1,k}|\]
and consequently
\[|hA| = \big(\text{number of $\equiv$-classes}\big) = \big(\text{number of $\cong$-classes}\big) = |\cX_{h,k}|-|\cX_{h-h_1,k}|,\]
as claimed.
\end{proof}

\subsection{Proof of Lemma~\ref{lem:all minima happen}}
\label{subsec:they happen}

\begin{proof}[Proof of Lemma~\ref{lem:all minima happen}]
Consider the set the set $A=\{0,1,(a-1)(b-1)+1,(a-1)(b-1)+a\}$, where $b\ge a\ge 2$. The $\cL$ lattice has first minimizer is $\vec{y_1}\coloneqq \langle 1-a,a-1,1,-1\rangle$ (with norm $2a$) and the second minimizer is $\vec{y_2}\coloneqq \langle 0,1,-b,b-1 \rangle$ (with norm $2b$). The resulting $\cL$ has minima $2a,2b$.

To see that the first and second minimizers are as claimed, we observe that the vector space perpendicular to $\vec 1$ and $\vec a$ has dimension 2, and $\vec{y_1},\vec{y_2}$ are linearly independent and therefore span the space in question (over the reals). We are then tasked with identifying all real numbers $\alpha,\beta$ with $\alpha\vec{y_1}+\beta\vec{y_2} \in \ZZ^4$. As
\[\alpha\vec{y_1}+\beta\vec{y_2}
  = \langle \alpha(1-a),   \alpha(a-1)+\beta,   \alpha-\beta b,   -\alpha+(b-1)\beta\rangle,\]
the first and second components being integers forces $\beta\in\ZZ$, from which the third component being integral forces $\alpha\in \ZZ$. Ergo, $\vec{y_1},\vec{y_2}$ are a \emph{lattice} basis for the lattice $\cL$.

Thus, the $L^1$-norm of a typical member of $\cL$ has the form
\begin{equation}\label{eq:normform}
| \alpha(1-a)| + |\alpha(a-1)+\beta| + |\alpha-\beta b| +|\alpha-(b-1)\beta|.
\end{equation}
For each fixed $\beta$, the norm form~\eqref{eq:normform} is a piecewise linear function of $\alpha$ whose minimum happens at one of the values
\[\alpha=0, \qquad \alpha=-\beta/(a-1),  \qquad \alpha=\beta b, \qquad  \alpha=(b-1)\beta.\]
Thus, the minimum norm is positive and is one of
  \[2b\beta ,  \qquad 2 \beta  \left(\frac{1}{a-1}+b\right),  \qquad 2 \beta  ((a-1) b+1), \qquad  2 \beta  ((a-1) (b-1)+1),\]
which is at least $2b$ in all cases. If $\beta=0$, without loss of generality $\alpha>0$, and the norm form simplifies to $2a\alpha$. This is minimized with $\alpha=1$, which gives us the first minimum of $2a$ and the first minimizer as $\vec{y_1}$. The second minimizer must be linearly independent of the first, i.e., $\beta\neq 0$, and from the above computation we now know that the second minimizer is $2b$, achieved by $\vec{y_2}$.

This proves Lemma~\ref{lem:all minima happen} for $k=4$. For larger $k$, we need only adjoin a sufficiently quickly growing sequence to $A$. As a specific instance of ``sufficiently quick'', it suffices to take $a_{i+1}>(h_2-1)\cdot a_i$ for $4\le i <k$.
\end{proof}

\subsection{Proof of Corollary~\ref{cor:they all happen}}
This is a consequence of Theorem~\ref{thm:main} with Lemma~\ref{lem:all minima happen}. With $b\ge a \ge 2$, the set $A=\{0,1,(a-1)(b-1)+1,(a-1)(b-1)+a\}\cup\{a_5,\dots,a_k\}$, where $a_{i+1}=2b\cdot a_i$ for $4\le i < k$  has first minimizer is $\langle 1-a,a-1,1,-1\rangle$ (with norm $2a$) and second minimizer is $\langle 0,1,-b,b-1 \rangle$ (with norm $2b$). The resulting coefficient lattice $\cL$ has minima $2a,2b$. For a given $h,j$ with $1\le j\le h$, choose any $b>h,b\ge a$ and $a=h-j+1\ge 2$ and Theorem~\ref{thm:main} now says
\[|hA| = \binom{h+k-1}{k-1}-\binom{h-a+k-1}{k-1}.\]

\section{Types of Tables}
\subsection{Some Notation for ``Additive'' Combinatorial Number Theory}
Let $G$ be a group (for example, $\RR$), and let $\varphi:G^h \to G^i$ be any function (for example, $\sigma_h(\vec x) = x_1+\dots+x_h$). Let $X\subseteq G$ (for example, $X=\NN$, the positive integers). If $i=1$, we define
\begin{align*}
\varphi X & \coloneqq \{\varphi(\vec g) : \vec g \in X^h\} \\
\cR_X(\varphi,k) & \coloneqq \{ |\varphi A | : |A|=k, A\subseteq X\} \\
m_X(\varphi,k) & \coloneqq \min \{|A| : A \subseteq X, |A|=k\} \\
M_X(\varphi,k) & \coloneqq \max \{|A| : A \subseteq X, |A|=k\}
\end{align*}

If $A \subseteq X$ has $|A|=k$ and $|\varphi A|=M_X(\varphi,k)$, then we call $A$ a $\varphi$-Sidon set. If $A \subseteq X$ has $|A|=k$ and $|\varphi A| = m_X(\varphi,k)$, then we call $A$ a $\varphi$-progression.

When $X$ is clear from the context, it is omitted from the notation. When $\varphi(\vec a)=a_1+\cdots+a_h$, we abbreviate it in the notation as simply $h$: the set $hA$ is the $h$-fold sumset of $A$ and $\cR_{\NN}(h,k)$ is the set of sizes of $h$-fold sumsets of sets of $k$ positive integers.

\[
\begin{array}{ccccccc}
\varphi(\vec a) & X & m_X(\varphi,k) & M_X(\varphi,k) & \cR_X(\varphi,k) & \varphi\text{-Sidon} & \varphi\text{-progression} \\ \hline
a_1+a_2 & \ZZ & 2k-1 & \binom{k+1}2 & [m,M] &\text{Sidon set} & \text{arithmetic prog.} \\
a_1-a_2 & \ZZ & 2k-1 & k^2-k+1 & \text{unknown} & \text{Sidon set} & \text{arithmetic prog.} \\
a_1 a_2 & \NN & 2k-1 & \binom{k+1}2 & [m,M] & \text{multiplicative Sidon set} & \text{geometric prog.} \\
a_1+a_2+a_3 & \ZZ & 3k-2 & \binom{k+2}{3} & \text{unknown} & B_3\text{-set} & \text{arithmetic prog.} \\
a_1+\dots+a_h & \ZZ & hk-h+1 & \binom{h+k-1}{k-1} & \text{unknown} & B_h\text{-set} & \text{arithmetic prog.}\\
a_1+2a_2 & \ZZ & 3k-2 & k^2 & \text{unknown} &\text{unnamed} & ???
\end{array}
\]

There are interesting questions with $i\ge 2$. For example, with $\varphi(x_1,x_2)=(x_1+x_2,x_1-x_2)$, one is led to questions about the relative sizes of sumsets and difference sets, for example~\cite{nathanson2007MSTD}. The work~\cite{2007authors5} concerns $\varphi(x_1,x_2)=(u_1x_1+u_2x_2, v_1x_1+v_2x_2)$ in full generality. The function $\varphi(x_1,x_2)=(x_1+x_2,x_1x_2)$ leads one to the infamous Sum-Product Conjecture~\cite{nathanson1997sumproduct,obryant2025visualizingsumproductconjecture}.

\subsection{Types of Tables}
If $G$ is an orderable group and $\varphi:G^h\to G$, we can denumerate the elements of $A$ as $\{a_1<\cdots<a_k\}$, and define the $\varphi$-table as the function $T$ from $[k]^h \to G$ with
\[T(\vec x) =\varphi(a_{x_1},a_{x_2},\dots,a_{x_h}).\]
For example, with $\varphi(\vec x)=x_1+x_2$, the table is $T(i,j)=a_i+a_j$.

We define the (weak) \emph{type} of the $\varphi$-table to be the equivalence relation on $[k]^h$ induced by $\varphi$, i.e., we have $\vec x \overset{\varphi,A}{\equiv} \vec y$ if and only if $T(\vec x)=T(\vec y)$.

We set $\cT_X(\varphi,k)$ to be the set of all types of $\varphi$-tables for $A\subseteq X$ with $|A|=k$.

We note that $|\varphi A|$ is the number of equivalence classes of the type of the $\varphi$-table, so that if $A$ and $B$ have the same type of $\varphi$-table, then $|\varphi A|=|\varphi B|$.

We comment that ``type'' as we have defined it here is somewhat weaker than the closely related ``type'' in~\cite{obryant2025visualizingsumproductconjecture}, and enumerated in~\seqnum{A378609}. There, ``type'' includes an ordering of the equivalence classes. The version in this work is easier to work with theoretically, while the additional information in the ordering is computationally useful.

\begin{question}
How many types of addition table are there? That is, what is $|\cT_{\ZZ}(h,k)|$?
\end{question}

\section{An Application of Types}
Nathanson~\cite{nathanson2025compressioncomplexitysumsetsizes} asked if $\cR_{\NN}(h,k)=\cR_{\RR}(h,k)$. We provide the stronger equality of types.
\begin{theorem}\label{thm:NvsR2}
$\displaystyle \cT_{\NN}(h,k)=\cT_{\RR}(h,k)$. In particular, $\cR_{\NN}(h,k)=\cR_{\RR}(h,k)$
\end{theorem}

\begin{definition}
For $\varphi:\RR^h \to \RR$, the separation of a set $A$ is
\[\sep_\varphi(X)\coloneqq \min \left\{\left|\phi(\vec a)-\phi(\vec b)\right| : \vec a,\vec b \in X^h, \varphi(\vec a) \neq \varphi(\vec b)\right\}.\]
If $\varphi(\vec x)=\sigma_h(\vec x) = x_1+\cdots+x_h$, we write ``$h$'' in place of ``$\varphi$'' in the notation, and note that $\sep_1(X)\ge \sep_2(X) \ge \cdots$.
\end{definition}

\begin{proof}[Proof of Theorem~\ref{thm:NvsR2}.]
For $h=1$ this is trivial, so we take $h\ge 2$. Let $X\subseteq\RR$ with $|X|=k$. We build a set $A\subseteq\NN_0$ with $|A|=k$ and $A,X$ having the same type of $h$-fold addition table. By affine invariance, we may assume that $X=\{0=x_1<\cdots < x_k=1\}$.
Let $q_0\in\NN$ be sufficiently large that $\sep_h(X)\cdot q_o \ge 2$.

Consider the sequence of vectors $d_q=\langle \{qq_0x_2\},\ldots,\{qq_0x_{k-1}\}\rangle$ of fractional parts of dilations of $x_2,\dots,x_{k-1}$. Our pigeons are $d_q$ for $0 \le q \le (k-2)^{2h}$. Our pigeonholes are the $(k-2)$-fold cartesian products of $\left[\frac i{2h},\frac{i+1}{2h}\right)$ with $0\le i < 2h$. We have $(k-2)^{2h}+1$ pigeons, and $(k-2)^{2h}$ pigeonholes, so there are $q',q''$ with
\[ 0 \le q' < q'' \le (k-2)^{2h},\]
and $d_{q'},d_{q''}$ in the same pigeonhole. That is, with $Q\coloneqq q''q_0-q'q_0\in[q_0,q_0(k-2)^{2h}]$ and $a_i$ the integer nearest $Q x_i$, we have
\[ -\frac 1{2h} < Q x_i -a_i < \frac 1{2h}, \qquad q_0 \le Q \le q_0(k-2)^{2h}\]
simultaneously for $1\le i \le k$.
Set $\epsilon_i\in(-1/(2h),1/(2h))$ so that $Qx_i=a_i+\epsilon_i$. Notice $\epsilon_1=\epsilon_k=0$.

We claim that $A=\{a_1,\dots,a_k\}$ has $|A|=|X|$ and $A,X$ have the same type of addition table. Since $x_1<\dots<x_k$ and $Q\ge q_0 \ge 1$, we clearly have $a_1\le \cdots \le a_k$. More precisely, we have \[a_{i+1}-a_i=Q(x_{i+1}-x_i)+(\epsilon_i-\epsilon_{i+1})
\ge \sep_1(X)Q -\frac1h \ge \sep_h(X)q_0 -\frac12 > 1\] so that $a_{i+1}>a_i$. In particular, $|A|=|X|$.

We will argue that the $h$-type of $X$ is the same as the $h$-type of $A$, whence $|hX|=|hA|$. First, assume that $(b_1,\dots,b_h)\overset{h,X}{\equiv} (c_1,\dots,c_h)$, so that
\[x_{b_1}+\dots+x_{b_h}=x_{c_1}+\dots+x_{c_h}.\]
One therefore has
\[a_{b_1}+\cdots a_{b_h}=a_{c_1}+\dots+a_{c_h}+\left(\epsilon_{c_1}+\dots+\epsilon_{c_h}\right)-\left(\epsilon_{b_1}+\dots+\epsilon_{b_h}\right).\]
From the defintion of $\epsilon_i$, we have
\[-1<\left(\epsilon_{c_1}+\dots+\epsilon_{c_h}\right)-\left(\epsilon_{b_1}+\dots+\epsilon_{b_h}\right)<1\]
and by the integrality of the $a_i$, this must be 0. That is,
\[a_{b_1}+\cdots a_{b_h}=a_{c_1}+\dots+a_{c_h}\]
and consequently $(b_1,\dots,b_h) \overset{h,A}{\equiv} (c_1,\dots,c_h)$.

Now, assume that $(b_1,\dots,b_h) \overset{h,A}{\equiv} (c_1,\dots,c_h)$, from which we have
\[a_{b_1}+\cdots a_{b_h}=a_{c_1}+\dots+a_{c_h}.\]
As $a_i = Qx_i-\epsilon_i$, we have
\[x_{b_1}+\dots+x_{b_h}=x_{c_1}+\dots+x_{c_h}- \left(\frac{\epsilon_{b_1}+\dots+\epsilon_{b_h}-\epsilon_{c_1}-\dots-\epsilon_{c_h}}{Q} \right).\]
As
\[\left| \frac{\epsilon_{b_1}+\dots+\epsilon_{b_h}-\epsilon_{c_1}-\dots-\epsilon_{c_h}}{Q} \right| < \frac{1}{Q} \leq \frac{\sep_h(X)}{2},\]
it must be that
\[x_{b_1}+\dots+x_{b_h}=x_{c_1}+\dots+x_{c_h}.\]
Thus, $(b_1,\dots,b_h) \overset{h,X}{\equiv} (c_1,\dots,c_h)$.
\end{proof}

\begin{theorem}
Let $\sigma(\vec x)=\sum_{i=1}^h x_i$ and $\pi(\vec x) = \prod_{i=1}^h x_i$. We have
\[\cT_{\NN}(\sigma,k) = \cT_{\NN}(\pi,k).\]
In particular, $\cR_{\NN}(\sigma,k) = \cR_{\NN}(\pi,k).$
\end{theorem}

\begin{proof}
Suppose that $S$ is a $k$-element set of positive integers, and let $P\coloneqq \{2^s : s \in S\}$. Then $a_1+\dots+a_h = b_1+\dots+b_h$ if and only if $2^{a_1}\cdots 2^{a_h}= 2^{b_1}\cdots 2^{b_h}$. Thus every type in $\cT_{\NN}(\sigma,k)$ is also in $\cT_{\NN}(\pi,k)$.

Now, suppose that $P$ is a set of positive integers with $|P|=k$. We must exhibit a set $A$ of positive integers whose $\sigma$-type is the same as the $\pi$-type of $P$.

Clearly, the $\pi$-type of $P$ is the same as the $\sigma$-type of $L_1 \coloneqq \{\log_2 p : p \in P\}$, but $L$ is not typically a set of integers. For $c>0$, we set $L_c \coloneqq \{c \log_2 p : p\in P\}$, which has the same $\sigma$-type as $L_1$ and $\pi$-type as $P$. We will show that it is possible to take $c$ so that $A_c \coloneqq \{c+[c \log_2 p] : p\in P\}\subseteq \NN$, where $[x]$ is the integer nearest to $x$, has the same $\sigma$-type as $L_c$.

If $c$ is sufficiently large, then clearly $A_c$ will be a set of $|P|$ distinct positive integers. As in the above argument, we can take $c$ both arbitrarily large and so that  $\{c+[c\log_2 p]: p \in P\}$ has arbitrarily large separation, and so that each $c+[c\log p]$ is within $1/(2h)$ of an integer. The proof goes through as above.

%\KOinline{This is kind of graceless. If the two proofs are the same, they should be written with one lemma that gets used twice.}
\end{proof}

\section{Experimental Thoughts}
\subsection{Nathanson's Experiment}\label{sec:NathansonExperiment}
Consider the following experiment conducted by Nathanson~\cite{nathanson2025triangulartetrahedralnumberdifferences}. We start by generating $10^7$ random subsets of $\{1,2,\dots,1000\}$ size $4$. We have $|10A|$ taking $213$ distinct values between $31$ and $286$, inclusive. Roughly $78\%$ of them are $B_{10}$-sets, i.e., $|hA|=\binom{10+4-1}{4-1}=286$. Of the $2\,202\,261$ other sets in our sample, we plot the proportion that have each possible size $|hA|$.
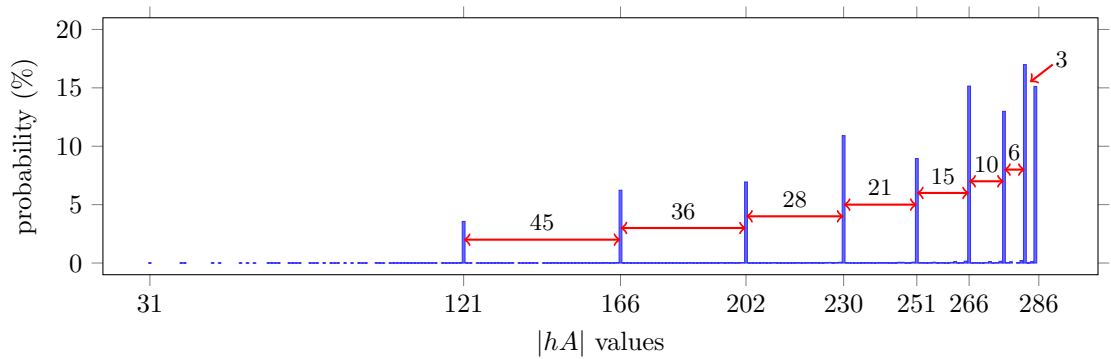
\begin{center}
\begin{tikzpicture}
\begin{axis}[
    width=0.95\linewidth,
    height=5cm,
    ybar,
    bar width=0.8,
    xlabel={$|hA|$ values},
    ylabel={probability ($\%$)},
    xmin=30.5,
    xmax=290,
    ymin=0,
    ymax=20,
    grid style=dashed,
    xtick={286, 266, 251, 230, 202, 166, 121, 31},
    tick align=outside,
    enlargelimits=0.05,
]

\addplot[fill=blue!60, draw=blue!80] coordinates {
(31,0.001862)
(40,0.002952)
(41,0.001181)
(49,0.005585)
(51,0.001181)
(57,0.005131)
(59,0.001408)
(61,0.0006811)
(65,0.005948)
(66,0.002906)
(67,0.001044)
(68,0.001862)
(71,0.001271)
(72,0.005721)
(73,0.0004087)
(74,0.003224)
(77,0.001226)
(78,0.001498)
(79,0.004041)
(81,0.007038)
(83,0.002679)
(84,0.001271)
(85,0.007084)
(87,0.001498)
(89,0.004813)
(91,0.006312)
(92,0.0009990)
(93,0.003451)
(96,0.01303)
(97,0.002180)
(98,0.0008173)
(100,0.0009990)
(101,0.006130)
(102,0.007538)
(103,0.0006357)
(104,0.001181)
(105,0.009309)
(106,0.002952)
(107,0.0006357)
(108,0.0007719)
(109,0.01281)
(110,0.002134)
(111,0.003042)
(112,0.006766)
(113,0.003814)
(115,0.01099)
(116,0.0009990)
(117,0.01081)
(118,0.002089)
(119,0.007674)
(120,0.01521)
(121,3.563)
(122,0.0009990)
(123,0.006039)
(125,0.001544)
(126,0.007765)
(127,0.001181)
(128,0.004268)
(129,0.001635)
(130,0.002588)
(131,0.01244)
(132,0.002543)
(133,0.001408)
(134,0.003042)
(135,0.01417)
(137,0.002997)
(138,0.002634)
(139,0.002452)
(140,0.01131)
(141,0.008400)
(142,0.0004995)
(144,0.005903)
(145,0.001226)
(146,0.01158)
(147,0.01671)
(148,0.007038)
(149,0.005131)
(150,0.01285)
(151,0.005494)
(152,0.001771)
(153,0.002952)
(154,0.01453)
(155,0.0009082)
(156,0.01480)
(157,0.003678)
(158,0.004223)
(159,0.01408)
(160,0.001226)
(161,0.01680)
(162,0.01848)
(163,0.005040)
(164,0.009263)
(165,0.02684)
(166,6.229)
(167,0.003042)
(168,0.01485)
(169,0.001635)
(170,0.002180)
(171,0.005086)
(172,0.002043)
(173,0.008900)
(174,0.005086)
(175,0.01103)
(176,0.002452)
(177,0.01167)
(178,0.001635)
(179,0.002815)
(180,0.01162)
(181,0.006039)
(182,0.01535)
(183,0.006493)
(184,0.001317)
(185,0.006266)
(186,0.02143)
(187,0.001953)
(188,0.002588)
(189,0.01303)
(190,0.01194)
(191,0.01158)
(192,0.01122)
(193,0.01231)
(194,0.007265)
(195,0.03024)
(196,0.003542)
(197,0.01648)
(198,0.02211)
(199,0.003042)
(200,0.02062)
(201,0.03891)
(202,6.939)
(203,0.002997)
(204,0.01367)
(205,0.0009082)
(206,0.006947)
(207,0.01753)
(208,0.002724)
(209,0.006766)
(210,0.03505)
(211,0.008945)
(212,0.004223)
(213,0.009490)
(214,0.007901)
(215,0.009399)
(216,0.01244)
(217,0.02275)
(218,0.001907)
(219,0.01539)
(220,0.02625)
(221,0.01494)
(222,0.02279)
(223,0.01308)
(224,0.005176)
(225,0.03515)
(226,0.04872)
(227,0.01208)
(228,0.03297)
(229,0.06530)
(230,10.91)
(231,0.02329)
(232,0.001771)
(233,0.003179)
(234,0.002906)
(235,0.01044)
(236,0.01153)
(237,0.02257)
(238,0.01358)
(239,0.002997)
(240,0.02007)
(241,0.02661)
(242,0.02157)
(243,0.01421)
(244,0.01035)
(245,0.01389)
(246,0.06475)
(247,0.05908)
(248,0.001635)
(249,0.01666)
(250,0.07315)
(251,8.946)
(252,0.02647)
(253,0.02130)
(254,0.001816)
(255,0.02843)
(256,0.06707)
(257,0.003769)
(258,0.02675)
(259,0.007765)
(260,0.002724)
(261,0.05208)
(262,0.1306)
(263,0.006266)
(264,0.02402)
(265,0.1597)
(266,15.15)
(267,0.001635)
(268,0.02384)
(269,0.005812)
(270,0.001317)
(271,0.04150)
(272,0.1235)
(273,0.0003179)
(274,0.03337)
(275,0.1478)
(276,12.99)
(277,0.03151)
(278,0.1073)
(280,0.02584)
(281,0.2152)
(282,16.99)
(283,0.006630)
(284,0.1157)
(285,15.12)
};

% Annotations for gap sizes
\node[above] at (axis cs:143.5,2) {\small 45};
\draw[<->, red, thick] (axis cs:121,2) -- (axis cs:166,2);

\node[above] at (axis cs:184,3) {\small 36};
\draw[<->, red, thick] (axis cs:166,3) -- (axis cs:202,3);

\node[above] at (axis cs:216,4) {\small 28};
\draw[<->, red, thick] (axis cs:202,4) -- (axis cs:230,4);

\node[above] at (axis cs:240.5,5) {\small 21};
\draw[<->, red, thick] (axis cs:230,5) -- (axis cs:251,5);

\node[above] at (axis cs:258.5,6) {\small 15};
\draw[<->, red, thick] (axis cs:251,6) -- (axis cs:266,6);

\node[above] at (axis cs:271,7) {\small 10};
\draw[<->, red, thick] (axis cs:266,7) -- (axis cs:276,7);

\node[above] at (axis cs:279,8) {\small 6};
\draw[<->, red, thick] (axis cs:276,8) -- (axis cs:282,8);

\node[above] at (axis cs:292.5,16) {\small 3};
\draw[->, red, thick] (axis cs:290,17) -- (axis cs:283.5,15.5);

\end{axis}
\end{tikzpicture}
\end{center}
\noindent We find that $97\%$ of the $2\,202\,261$ non-$B_{10}$-sets\footnote{It is expected that most of our sample are $B_{10}$-sets: with fixed $h,k$, if $n$ is sufficiently large then almost all sets in $\binom{[n]}{k}$ are $B_h$-sets.} have $|hA|$ among the $9$ sizes
\[285, 282, 276, 266, 251, 230, 202, 166, 121 .\]
The first differences of the list of common values of $|10A|$ that are less than $286$ are
\[45,36,28,21,15,10,6,3,\]
i.e., the triangular numbers. It is this mystery\,---\,the appearance of triangular numbers\,---\,to which we have given a partial explanation in this work.

We note that the partial sums of the triangular numbers are the tetrahedral numbers given by $\binom{t+2}{3}$. Nathanson's experiment indicates that a random $4$-element subset $A$ of $[n]$ is highly likely to have $|hA|=\binom{h+3}{3}$, but if it doesn't then it is highly likely to have $|hA|$ among $\binom{h+3}{3}-\binom{t+2}{3}$, with $1\le t < h$.

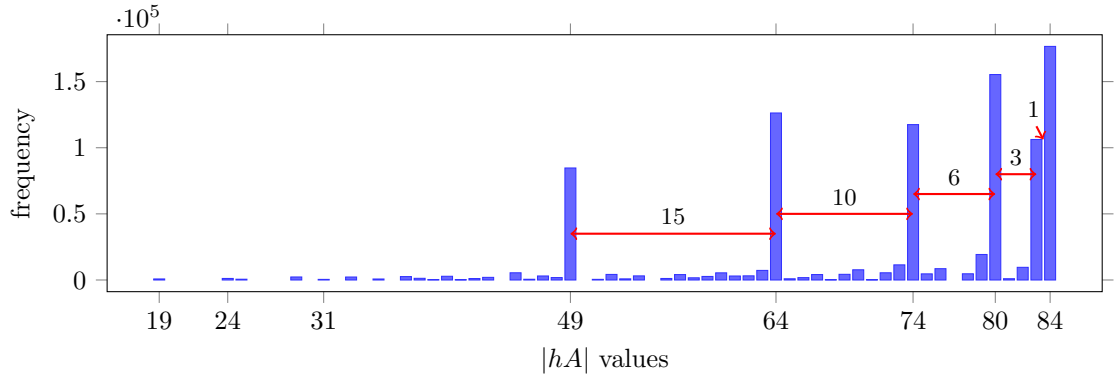
\begin{figure}
\begin{center}
\begin{tikzpicture}
\begin{axis}[
    width=0.95\linewidth,
    height=5cm,
    ybar,
    bar width=0.8,
    xlabel={$|hA|$ values},
    ylabel={frequency},
    xmin=18.5,
    xmax=84.5,
    ymin=0,
    %ymajorgrids=true,
    grid style=dashed,
    xtick={84,80,74,64,49,19,24,31},
    tick align=outside,
    enlargelimits=0.05,
]

\addplot[fill=blue!60, draw=blue!80] coordinates {
(19, 782)
(24, 1156)
(25, 578)
(29, 2275)
(31, 455)
(33, 2244)
(35, 748)
(37, 2579)
(38, 1260)
(39, 315)
(40, 2806)
(41, 272)
(42, 1088)
(43, 1981)
(45, 5444)
(46, 544)
(47, 3031)
(48, 1836)
(49, 84693)
(51, 476)
(52, 4200)
(53, 840)
(54, 3098)
(56, 1138)
(57, 4090)
(58, 1648)
(59, 2678)
(60, 5398)
(61, 3010)
(62, 3130)
(63, 7240)
(64, 126278)
(65, 900)
(66, 1780)
(67, 4120)
(68, 280)
(69, 4276)
(70, 7712)
(71, 260)
(72, 5444)
(73, 11432)
(74, 117496)
(75, 4624)
(76, 8548)
(78, 4686)
(79, 19280)
(80, 155350)
(81, 954)
(82, 9570)
(83, 106252)
(84, 176620)
};

% Annotations for gap sizes
\node[above] at (axis cs:82.8,116000) {\small 1};
\draw[->, red, thick] (axis cs:83,116252) -- (axis cs:83.5,106252);

\node[above] at (axis cs:81.5,80000) {\small 3};
\draw[<->, red, thick] (axis cs:80,80000) -- (axis cs:83,80000);

\node[above] at (axis cs:77,65000) {\small 6};
\draw[<->, red, thick] (axis cs:74,65000) -- (axis cs:80,65000);

\node[above] at (axis cs:69,50000) {\small 10};
\draw[<->, red, thick] (axis cs:64,50000) -- (axis cs:74,50000);

\node[above] at (axis cs:56.5,35000) {\small 15};
\draw[<->, red, thick] (axis cs:64,35000) -- (axis cs:49,35000);

\end{axis}
\end{tikzpicture}
\end{center}
\caption{The bar chart of the $6$-fold sumset sizes of all $4$-element subsets $A\subseteqq [70]\coloneqq \{1,2,\dots,70\}$. As Nathanson observed, the distribution is strongly concentrated on sizes of the form $\binom{h+3}{3}-\binom{t+2}{3}$ for $0\le t < h$.}
\end{figure}

\subsection{Distribution of Minima}
This author believes that the appearance of triangular and tetrahedral numbers in Nathanson's experiment is a consequence of Theorem~\ref{thm:main} and the as-yet-unproven reality that roughly $80\%$ of $\binom{[1000]}{4}$ have $h_1> 6$, and of the remaining $20\%$, almost all have $h_2>6$. More generally, the following conjecture combined with Theorem~\ref{thm:main} would be a satisfactory explanation.
\begin{conjecture}
Fix $h\ge3$, and let $\epsilon>0$. There is an $n_0$ such that for all $n\ge n_0$, at least $(1-\epsilon)\binom{n}{4}$ of the subsets of $\binom{[n]}{4}$ have $h_1> h$, and of the sets with $h_1\le h$, the proportion with $h_2\le h$ is at least $1-\epsilon$.
\end{conjecture}

This is really a question about a ``random lattice'', the lattice perpendicular to $\vec 1$ and $\vec a$. Experiments indicate that if $A$ is chosen uniformly from $\binom{[n]}{4}$, then the expected value of $h_1(A)$ is close to $\frac12 n^{1/2}$, with standard deviation around $\frac15 n^{1/2}$. If $A$ is chosen uniformly from $\binom{[n]}{5}$, the expected value seems to approach $\frac58 n^{1/3}$, with the standard deviation in the neighborhood of $\frac15 n^{1/3}$.

Example: Let $A=\{1,5,96,100\}$. The set $A$ is not a $B_2$-set as $1+100=5+96$. In lattice terms, $h_1=2$ and $\vec y_1 = \langle 1,-1,-1,1\rangle$. However, $A$ does have
\[|hA| = \binom{h+3}{3}-\binom{h+1}{3}\]
for $2 \le h \le 94$. For $h=95$, though, we encounter the new equality:
\[ 95\cdot 96 = 91\cdot 100+4\cdot 5.\]
In lattice terms, $h_2=95$ and $y_2=\langle 0,-4,95,-91\rangle$.

This experimental work would benefit from a positive answer to the question:
\begin{question}
Is there a way to generate $k$-element subsets of $[n]$ that are not $B_h$-sets uniformly that is more efficient than simply generating $k$-element subsets of $[n]$ and testing for the $B_h$-property?
\end{question}

\subsection{Nathanson's Program}
In recent works, Nathanson~\cite{nathanson2025compressioncomplexitysumsetsizes,nathanson2025explicitsumsetsizesadditive,nathanson2025inverseproblemssumsetsizes,nathanson2025problemsadditivenumbertheory,nathanson2025triangulartetrahedralnumberdifferences} has given a program to develop a deeper understanding of $|hA|$, beyond merely the extreme values. The current work fits within his framework.

\paragraph{A gap in the literature.}
In~\cite{nathanson2025compressioncomplexitysumsetsizes}, Nathanson writes ``it is not known if $\cR_{\NN_0}(h,k)= \cR_{\RR}(h,k)$.'' Our Theorem~\ref{thm:NvsR} fills this gap in the literature.

\paragraph{A Problem from Nathanson.}
In~\cite{nathanson2025triangulartetrahedralnumberdifferences}, Problem 8 is to prove that
\[\left\{ \binom{h+3}{3}-\binom{j+2}{3} : 1\le j \le h\right\} \subseteq \cR_{\ZZ}(h,4).\]
This is a consequence of Theorem~\ref{thm:main} with the set $A=\{0,1,(a-1)(b-1)+1,(a-1)(b-1)+a\}$, where $b\ge a\ge 2$. The first minimizer is $\langle 1-a,a-1,1,-1\rangle$ (with norm $2a$) and the second minimizer is $\langle 0,1,-b,b-1 \rangle$ (with norm $2b$). The resulting $\cL$ has minima $2a,2b$. For a given $h,j$ with $1\le j\le h$, choose any $b>h$ and $a=h-j+1$ and Theorem~\ref{thm:main} now says
\[|hA| = \binom{h+3}{3}-\binom{h-a+3}{3}=\binom{h+3}{3}-\binom{j+2}{3}.\]

\paragraph{Another Problem from Nathanson.}
In~\cite{nathanson2025triangulartetrahedralnumberdifferences}, Problem 9 is to ``Discover the pattern of popular sumset sizes for sets of size $k\ge 5$.'' The current work suggests that the popular values will be
\[\left\{ \binom{h+k-1}{k-1} - \binom{j+k-1}{k-1} : 0\le j < h\right\},\]
provided that $n$ is sufficiently large in terms of $h,k$. A modest amount of experimentation agrees with this suggestion.

\section*{Acknowledgements}
This work was inspired by Melvyn B. Nathanson's many articles on the sequence of sumset sizes, a few of the most recent of which are cited in the bibliography. Moreover, the prose has benefited from his questions and suggestions.

%\bibliographystyle{plain}
%\bibliography{ToT}

\end{document}